\documentclass{amsproc}

\newcommand*{\KAI}{\CJKfamily{kai}}
\usepackage{amsthm}
\usepackage{amsmath,amssymb,amsfonts}
\usepackage{indentfirst,latexsym,bm}
\usepackage{CJK}
\usepackage{fancyhdr}
\usepackage[english]{babel}
\usepackage{epsfig}
\usepackage{lmodern}
\usepackage[T1]{fontenc}
\usepackage{times}
\usepackage{multicol}
\usepackage{multimedia}
\usepackage[all]{xy}
\newcommand{\kai}{\CJKfamily{kai}}      
\newcommand{\hei}{\CJKfamily{hei}}      

\newtheorem{dingyi}{\hei Definition~}[section]
\newtheorem{dingli}{\hei Theorem~}[section]
\newtheorem{yinli}[dingli]{\hei Lemma~}
\newtheorem{tuilun}[dingli]{\hei Corollary~}
\newtheorem{mingti}[dingli]{\hei Proposition~}

\theoremstyle{definition}

\theoremstyle{remark}

\numberwithin{equation}{section}

\begin{document}
\setlength{\parindent}{2em}                 
\setlength{\parskip}{3pt plus1pt minus1pt}  
\setlength{\abovedisplayskip}{2pt plus1pt minus1pt}     
\setlength{\belowdisplayskip}{6pt plus1pt minus1pt}     


\title{On the Center of Two-parameter Quantum Groups $U_{r,
s}(\mathfrak{so}_{2n+1})$}

\author[Hu]{Naihong Hu$^\star$}
\address{Department of Mathematics, Shanghai Key Laboratory of Pure Mathematics and Mathematical
Practice, East China Normal University,
Min Hang Campus, Dong Chuan Road 500, Shanghai 200241, PR China}
\email{nhhu@math.ecnu.edu.cn}
\thanks{$^\star$N.H., Corresponding author, supported in part by the NNSF (Grant: 11271131), the National \& Shanghai Leading
Academic Discipline Projects (Project Number: B407).}

\author[Shi]{Yuxing Shi}
\address{Department of Mathematics, Shanghai Key Laboratory of Pure Mathematics and Mathematical
Practice, East China Normal University,
Min Hang Campus, Dong Chuan Road 500, Shanghai 200241, PR China}\email{52100601007@student.ecnu.cn}

\subjclass{Primary 17B37, 81R50; Secondary 17B35}
\date{October 22, 2013}


\keywords{two-parameter quantum group, Harish-Chandra
homomorphism, Rosso form}
\begin{abstract}
The paper mainly considers the center of two-parameter quantum groups $U_{r,s}(\mathfrak{so}_{2n+1})$ via an analogue
of the Harish-Chandra homomorphism. In the case when $n$ is odd, the Harish-Chandra homomorphism is not injective
in general. When $n$ is even, the Harish-Chandra homomorphism is injective and  the center of
two-parameter quantum groups $U_{r,s}{(\mathfrak{so}_{2n+1})}$ is described, up to isomorphism.
\end{abstract}

\maketitle
\section{Introduction}
From down-up algebras approach (\cite{BW1}), Benkart-Witherspoon (\cite{BW2})  recovered Takeuchi's definition of two-parameter quantum groups of type $A$. Since then, a
systematic study of the two-parameter quantum groups has been going
on, see, for instance, (\cite{BW2}, \cite{BW3}) for type $A$; (\cite{BGH1}, \cite{BGH2}) for types $B, C, D$; \cite{HS}, \cite{BH} for types $G_2, E_6, E_7,
E_8$. For a unified definition, see (\cite{HP1}, \cite{HP2}). In this paper we determine the center of the two-parameter quantum groups of type $B_n$ when $n$ is even.

Much work has been done on the center of quantum groups for finite-dimensional simple Lie algebras (\cite{B}, \cite{D}, \cite{JL}, \cite{R}, \cite{RTF}, \cite{T}), and also for (generalized) Kac-Moody (super)algebras (\cite{E}, \cite{Hong}, \cite{KT}), and for two-parameter quantum group of type $A$ \cite{BKL}. The approach taken in many of these papers (and adopted here as well) is to define a bilinear form on the quantum group which is invariant under the adjoint action. This quantum version of the Killing form is often referred to in the one-parameter setting as the Rosso form (see \cite{R}). The next
step involves constructing an analogue $\xi$ of the Harish-Chandra map. It is straightforward to show that the map $\xi$ is an injective algebra homomorphism. The main difficulty lies in determining the image of $\xi$ and in finding enough central elements to prove that the map $\xi$ is surjective. In the two-parameter case for type $A_n$ \cite{BKL}, a new phenomenon arises: the $n$ odd and even cases behave differently. Additional central elements arise when $n$ is even. As for type $B_n$, we will show that $\xi$ is an injective map only when $n$ is even.

The paper is organized as follows. In section 2, the definition of two-parameter quantum group $U_{r,
s}(\mathfrak{so}_{2n+1})$ and some related properties are given. In section 3, we introduce a Harish-Chandra homomorphism $\xi$, and prove that $\xi$ is injective when $n$ is even. In section 4, we determine the center of two-parameter quantum groups $U_{r,s}{(\mathfrak{so}_{2n+1})}$ in the sense of isomorphisms when $n$ is even.

\section{$U_{r,s}(\mathfrak{so}_{2n+1})$ and related properties}
\subsection{Definition of $U_{r,s}(\mathfrak{so}_{2n+1})$}
The definition of two-parameters quantum groups of type $B$ was given in ~\cite{{BGH1}}. Here $\mathbb{K}=\mathbb{Q}(r,s)$ is a subfield of $\mathbb C$ with $r$, $s\in \mathbb{C}$, $r^{2}+s^{2}\neq
 1$, $r\neq s$. $\Phi$ is the root system of $\mathfrak{so}_{2n+1}$ with $\Pi$ a base of simple roots,~which is a finite subset of a Euclidean space $E=\mathbb{R}^n$ with an inner product $( , )$.~
 Let $\varepsilon_{1},\cdots,\varepsilon_{n}$~denote a normal orthogonal basis of ~$E$,~then
 $\Pi=\{\alpha_{i}=\varepsilon_{i}-\varepsilon_{i+1},1\leq i\leq n-1,\alpha_{n}=\varepsilon_{n}\}$,
 $\Phi=\{\pm\varepsilon_{i}\pm\varepsilon_{j},~1\leq i<j\le n;~\pm\varepsilon_{j},~1\leq j\leq n\}$.
Denote $r_{i}=r^{(\alpha_{i},\alpha_{i})}$,
~$s_{i}=s^{(\alpha_{i},\alpha_{i})}.$

\begin{dingyi}
\begin{upshape}
Let $U=U_{r, s}(\mathfrak{so}_{2n+1})$ be the associative algebra over $\mathbb{K}$ generated by $e_{i},~f_{i},~
\omega^{\pm1}_{i},~\omega^{\prime\pm1}_{i}(i=1,\cdots,n)$, subject to relations \upshape{$(B1)$---$(B7)$:}
\begin{itemize}
\item[$(B1)$] $\omega^{\pm1}_{i}, \omega^{\prime\pm1}_{j}$ all commute with one another and $\omega_{i}\omega^{-1}_{i}=1=\omega^\prime_{j}\omega^{\prime-1}_{j}$ for $1\leq i, j\leq n$.
\item[$(B2)$] For $1\leq i\leq n $,~$1\leq j<n$,~there are the following identities:
\begin{eqnarray*}
\omega_{j}e_{i}\omega^{-1}_{j}=r^{(\varepsilon_{j},
\alpha_{i})}_{j}s^{(\varepsilon_{j+1}, \alpha_{i})}_{j}e_{i},
&\quad\quad&
\omega_{j}f_{i}\omega^{-1}_{j}=r^{-(\varepsilon_{j}, \alpha_{i})}_{j}s^{-(\varepsilon_{j+1}, \alpha_{i})}_{j}f_{i}, \\
\omega_{n}e_{j}\omega^{-1}_{n}=r^{2(\varepsilon_{n},
\alpha_{j})}_{n}e_{j}, &\quad\quad&
\omega_{n}f_{j}\omega^{-1}_{n}=r^{-2(\varepsilon_{n}, \alpha_{j})}_{n}f_{j}, \\
\omega_{n}e_{n}\omega^{-1}_{n}=r^{(\varepsilon_{n},
\alpha_{n})}_{n}s^{-(\varepsilon_{n}, \alpha_{n})}_{n}e_{n},
&\quad\quad& \omega_{n}f_{n}\omega^{-1}_{n}=r^{-(\varepsilon_{n},
\alpha_{n})}_{n}s^{(\varepsilon_{n}, \alpha_{n})}_{n}f_{n}.
\end{eqnarray*}
\item[$(B3)$] For $1\leq i\leq n$, $1\leq j<n$,~there are the following identities:
\begin{eqnarray*}
\omega^\prime_{j}e_{i}\omega^{\prime-1}_{j}=s^{(\varepsilon_{j},
\alpha_{i})}_{j}r^{(\varepsilon_{j+1}, \alpha_{i})}_{j}e_{i},
&\quad\quad&
\omega^\prime_{j}f_{i}\omega^{\prime-1}_{j}=s^{-(\varepsilon_{j}, \alpha_{i})}_{j}r^{-(\varepsilon_{j+1}, \alpha_{i})}_{j}f_{i}, \\
\omega^\prime_{n}e_{j}\omega^{\prime-1}_{n}=s^{2(\varepsilon_{n},
\alpha_{j})}_{n}e_{j}, &\quad\quad&
\omega^\prime_{n}f_{j}\omega^{\prime-1}_{n}=s^{-2(\varepsilon_{n}, \alpha_{j})}_{n}f_{j}, \\
\omega^\prime_{n}e_{n}\omega^{\prime-1}_{n}=s^{(\varepsilon_{n},
\alpha_{n})}_{n}r^{-(\varepsilon_{n}, \alpha_{n})}_{n}e_{n},
&\quad\quad&
\omega^\prime_{n}f_{n}\omega^{\prime-1}_{n}=s^{-(\varepsilon_{n},
\alpha_{n})}_{n}r^{(\varepsilon_{n}, \alpha_{n})}_{n}f_{n}.
\end{eqnarray*}
\item[$(B4)$] For $1\leq i, j\leq n$,~there are the following identities:
\begin{displaymath}
[e_{i},
f_{j}]=\delta_{ij}\frac{\omega_{i}-\omega^{\prime}_{i}}{r_{i}-s_{i}}.
\end{displaymath}
\item[$(B5)$] For~$|i-j|>1$,~there are $(r,s)$-Serre relations:
\begin{displaymath}
[e_{i}, e_{j}]=[f_{i}, f_{j}]=0.
\end{displaymath}

\item[$(B6)$] For~$1\leq i<n, 1\leq j<n-1$,~there are $(r,s)$-Serre relations:
$$
e^2_{i}e_{i+1}-(r_{i}+s_{i})e_{i}e_{i+1}e_{i}+(r_{i}s_{i})e_{i+1}e^2_{i}=0,
$$
$$
e^2_{j+1}e_{j}-(r^{-1}_{j+1}+s^{-1}_{j+1})e_{j+1}e_{j}e_{j+1}+(r^{-1}_{j+1}s^{-1}_{j+1})e_{j}e^2_{j+1}=0,
$$
$e^3_{n}e_{n-1}-(r^{-2}_{n}+r^{-1}_{n}s^{-1}_{n}+s^{-2}_{n})e^2_{n}e_{n-1}e_{n}$\\
$=-(r^{-1}_{n}s^{-1}_{n})(r^{-2}_{n}+r^{-1}_{n}s^{-1}_{n}+s^{-2}_{n})e_{n}e_{n-1}e^2_{n}+(r^{-3}_{n}s^{-3}_{n})e_{n-1}e^3_{n}.$

\item[$(B7)$] For $1\leq i<n, 1\leq j<n-1$,~there are $(r,s)$-Serre relations:
$$
f_{i+1}f^2_{i}-(r_{i}+s_{i})f_{i}f_{i+1}f_{i}+(r_{i}s_{i})f^2_{i}f_{i+1}=0,
$$
$$
f_{j}f^2_{j+1}-(r^{-1}_{j+1}+s^{-1}_{j+1})f_{j+1}f_{j}f_{j+1}+(r^{-1}_{j+1}s^{-1}_{j+1})f^2_{j+1}f_{j}=0,
$$
$f_{n-1}f^3_{n}-(r^{-2}_{n}+r^{-1}_{n}s^{-1}_{n}+s^{-2}_{n})f_{n}f_{n-1}f^2_{n}$\\
$=-(r^{-1}_{n}s^{-1}_{n})(r^{-2}_{n}+r^{-1}_{n}s^{-1}_{n}+s^{-2}_{n})f^2_{n}f_{n-1}f_{n}+(r^{-3}_{n}s^{-3}_{n})f^3_{n}f_{n-1}.$
\end{itemize}

The Hopf algebra structure on $U_{r,s}(\mathfrak{so}_{2n+1})$ with the comultiplication,
the counit and the antipode as follows
\begin{eqnarray*}
\Delta(\omega^{\pm1}_{i})=\omega^{\pm1}_{i}\otimes \omega^{\pm1}_{i}
, &\quad\quad& \Delta(\omega^{\prime\pm1}_{i})=\omega^{\prime\pm1}_{i}\otimes\omega^{\prime\pm1}_{i}, \\
\Delta(e_{i})=e_{i}\otimes1+\omega_{i} \otimes
 e_{i}, &\quad\quad& \Delta(f_{i})=1\otimes f_{i}+f_{i}\otimes\omega^{\prime}_{i}, \\
 \varepsilon(\omega^{\pm1}_{i})=\varepsilon(\omega^{\prime\pm1}_{i})=1, &\quad\quad&
 \varepsilon(e_{i})= \varepsilon(f_{i})=0, \\
 S(\omega^{\pm1}_{i})=\omega^{\mp1}_{i}, &\quad\quad&
 S(\omega^{\prime\pm1}_{i})=\omega^{\prime\mp1}_{i}, \\
 S(e_{i})=-\omega^{-1}_{i}e_{i}, &\quad\quad&
 S(f_{i})=-f_{i}\omega^{\prime-1}_{i}.
\end{eqnarray*}
\end{upshape}
\end{dingyi}
Let $\Lambda=\bigoplus^n_{i=1}\mathbb{Z}\varpi_{i}$ be the weight lattice of $\mathfrak{so}_{2n+1}$, where $\varpi_i$ are the fundamental weights. Let $\Lambda^+=\{\lambda\in\Lambda\mid(\alpha_i,\lambda)\geq 0,~1\leq i\leq n\}$ denote the set of dominant weights for $\mathfrak{so}_{2n+1}$.~$Q=\bigoplus_{i=1}^n\mathbb{Z}\alpha_i$ denote the root lattice and set $Q^+=\bigoplus_{i=1}^n\mathbb{Z}_{\geq 0}\alpha_i$.

$U$ has a triangular decomposition $U\cong U^-\otimes U_0\otimes U^+$, where $U_0$ is the subalgebra generated by $\omega^{\pm}_{i},~\omega^{\prime\pm}_{i}$, and $U^+$ (resp. $U^-$) is the subalgebra generated
by $e_i$ (resp. $f_i$). Let $\mathcal{B}$ (resp. $\mathcal{B}^\prime$) denote the Hopf subalgebra of $U$ generated by $e_j$, $\omega^{\pm}_{j}$ (resp. $f_j$, $\omega^{\prime\pm}_{j}$)
with $1\leq j\leq n$.

\begin{mingti}\cite{BGH1}
\begin{upshape}
There exists a unique skew-dual pairing $\langle\,,\,\rangle$:$\
\mathcal{B}^\prime\times\mathcal{B}\rightarrow
\mathbb{K}$ of the Hopf subalgebras $\mathcal{B}$ and $\mathcal{B}^\prime$ such that
\begin{eqnarray*}
\langle f_{i}, e_{j}\rangle=\delta_{ij}\frac{1}{s_{i}-r_{i}},
&\quad\quad& 1\leq
i, j\leq n, \\
\langle\omega^\prime_{i}, \omega_{j}\rangle=r^{2(\varepsilon_{j},
\alpha_{i})}s^{2(\varepsilon_{j+1}, \alpha_{i})},
&\quad\quad& 1\leq i\leq n, 1\leq j< n, \\
\langle\omega^\prime_{i}, \omega_{n}\rangle=r^{2(\varepsilon_{n},
\alpha_{i})},
&\quad\quad&1\leq i< n, \\
\langle\omega^{\prime\pm1}_{i},
\omega^{-1}_{j}\rangle=\langle\omega^{\prime\pm1}_{i},
\omega_{j}\rangle^{-1}= \langle\omega^{\prime}_{i},
\omega_{j}\rangle^{\mp1}, &\quad\quad&1\leq
i, j\leq n, \\
\langle\omega^\prime_{n}, \omega_{n}\rangle=rs^{-1}.
\end{eqnarray*}
and all other pairs of generators are 0. Moreover, we have
$\langle S(a), S(b)\rangle=\langle a, b\rangle$ for $a\in\mathcal{B}^\prime, b\in\mathcal{B}$.
\end{upshape}
\end{mingti}

\begin{tuilun}\cite{BGH1}
\begin{upshape}
For $\varsigma=\sum^n_{i=1}\varsigma_{i}\alpha_{i}\in Q$, the defining relations $(B3)$ of  $U_{r, s}(\mathfrak{so}_{2n+1})$ can be written as
\begin{eqnarray*}
\omega_{\varsigma}e_{i}\omega^{-1}_{\varsigma}=\langle\omega^{\prime}_{i},
\omega_{\varsigma}\rangle e_{i}, &\quad\quad&
\omega_{\varsigma}f_{i}\omega^{-1}_{\varsigma}=\langle\omega^{\prime}_{i}, \omega_{\varsigma}\rangle^{-1}f_{i}, \\
\omega^{\prime}_{\varsigma}e_{i}\omega^{\prime-1}_{\varsigma}=\langle\omega^{\prime}_{\varsigma},
\omega_{i}\rangle^{-1}e_{i}, &\quad\quad&
\omega^{\prime}_{\varsigma}f_{i}\omega^{\prime-1}_{\varsigma}=\langle\omega^{\prime}_{\varsigma}, \omega_{i}\rangle f_{i}.
\end{eqnarray*}
\end{upshape}
\end{tuilun}
Corresponding to any $\lambda\in Q$ is an
algebra homomorphism $\varrho^\lambda:U_0\rightarrow \mathbb{K}$
given by $\varrho^\lambda(\omega_{j})=\langle\omega^\prime_{\lambda},~\omega_{j}\rangle $,~$\varrho^\lambda(\omega^\prime_{j})=\langle \omega^\prime_{j},~\omega_{\lambda}\rangle^{-1}$.
Associated with any algebra homomorphism $\psi: U_0\rightarrow\mathbb{K}$ is the Verma module $M(\psi)$ with highest weight $\psi$ and its unique irreducible quotient $L(\psi)$. When the
highest weight is given by the homomorphism $\varrho^\lambda$ for $\lambda\in\Lambda$, we simply write $M(\lambda)$
and $L(\lambda)$ instead of $M(\varrho^\lambda)$ and $L(\varrho^\lambda)$.~They all belong to category $\mathcal{O}$, for more details, please refer to \cite{BGH2}.

\begin{yinli}\cite{BGH2}
\begin{upshape}
Let $\upsilon_{\lambda}$ be a highest weight vector of
$M(\lambda)$ for $\lambda\in\Lambda^{+}$. Then the irreducible
module $L(\lambda)$ is  given by
\begin{displaymath}
L(\lambda)=L^\prime(\lambda)= M(\lambda)\bigg/
\Big(\sum\limits_{i=1}^{n}Uf^{(\lambda,
\alpha^{\vee}_{i})+1}_{i}\upsilon_{\lambda}\Big).
\end{displaymath}
\end{upshape}
\end{yinli}
Let $W$ be the Weyl group of the root system $\Phi$, and $\sigma_i\in W$ the reflection associated to $\alpha_i$ for each $1\leq i\leq n$. Thus,
$\sigma_{i}(\lambda)=\lambda-(\lambda, \alpha_{i})\alpha_{i},\
\lambda\in\Lambda$.~
By \cite{BGH2}, $L(\lambda)$ is a finite-dimensional $U$-module on which $U_{0}$ acts semi-simply.
$L(\lambda)=\bigoplus\limits_{\mu\leq\lambda} L(\lambda)_{\mu}$,
where $\varrho^\mu:U_{0}\rightarrow\mathbb{K}$ is an algebra homomorphism, and
$L(\lambda)_{\mu}=\{x\in
L(\lambda)\mid\omega_{i}.x=\varrho^{\mu}(\omega_{i})x,~ \omega^{\prime}_{i}.x=\varrho^{\mu}(\omega^{\prime}_{i})x,~ 1\leq
 i\leq n\}.$

\begin{yinli}\cite{BGH2}\label{BGH1Lemma215}
$(a)$ The elements $e_i$, $f_i(1\leq i \leq n)$ act
locally nilpotently on $U_{r,s}(\mathfrak{so}_{2n+1})$-module $L^\prime(\lambda)$.

$(b)$ Assume that $rs^{-1}$ is not a root of unity, $V=\bigoplus_{j\in \mathbb{Z}^+}V_{\lambda-j\alpha}\in Ob(\mathcal{O})$ is a $U_{r,s}(\mathfrak{sl}_2)$-module for some $\lambda\in\Lambda$. If $e$, $f$ act locally nilpotently on $V$, then $\dim_{\mathbb{K}}V<\infty$,
         and the weights of $V$ are preserved
under the simple reflection taking $\alpha$ to $-\alpha$.
\end{yinli}

\begin{dingli}
If $\lambda\in\Lambda^{+}$,~then~
$\dim L(\lambda)_{\mu}=\dim L(\lambda)_{\sigma(\mu)}$,
~$\forall\mu\in\Lambda,~\sigma\in W$.
\end{dingli}
\begin{proof}[\upshape\KAI Proof.]
$L(\lambda)= L^{\prime}(\lambda)$,~$U_{i}$ a subalgebra generated by
$e_{i}, f_{i},\omega_{i}, \omega^{\prime}_{i}$,
$\mu$ is a weight of $L^{\prime}(\lambda)$,
$
L^{\prime}_{i}(\mu)=U_{i}L^{\prime}(\lambda)_{\mu}
=\sum\limits_{j\in
\mathbb{Z}^{+}}L^{\prime}_{i}{(\mu)}_{\lambda^{\prime}-j\alpha_{i}},
$
where $\lambda^{\prime}\leq \lambda$. By \cite{BGH2} Lemma 2.15,
simple reflection $\sigma_{i}$ preserves weights of $L^{\prime}_{i}(\mu)$,
$\sigma_{i}(\mu)$ is weight of both $L^{\prime}_{i}(\mu)$ and $L^{\prime}(\lambda)$.

Since
$\dim L^{\prime}(\lambda)_{\mu}\geq \dim
L^{\prime}(\lambda)_{\sigma_{i}(\mu)}\geq \dim
L^{\prime}(\lambda)_{\mu}$,
~$\dim L^{\prime}(\lambda)_{\mu}=\dim
L^{\prime}(\lambda)_{\sigma_{i}(\mu)}$.
So $\dim
L^{\prime}(\lambda)_{\mu}= \dim L^{\prime}(\lambda)_{\sigma(\mu)}$.
\end{proof}

\begin{dingyi}\cite{BGH1}
\begin{upshape}\kai
Bilinear form $\langle~,~\rangle_{U}: U\times U \rightarrow
\mathbb{K}$ defined by
\begin{eqnarray*}
\langle F_{a}\omega^\prime_{\mu}\omega_{\upsilon}E_{\beta},
F_{\theta}\omega^\prime_{\sigma}\omega_{\delta}E_{\gamma}\rangle_{U}
=\langle \omega^\prime_{\sigma}, \omega_{\upsilon}\rangle
\langle\omega^\prime_{\mu}, \omega_{\delta}\rangle \langle
F_{\theta}, E_{\beta}\rangle \langle S^2(F_{a}), E_{\gamma}\rangle.
\end{eqnarray*}This form is also called the Rosso form of the two-parameter quantum group $U_{r, s}(\mathfrak{g})$.
\end{upshape}
\end{dingyi}

\begin{dingli}\cite{BGH1}
\begin{upshape} The Rosso form on $U_{r, s}(\mathfrak{g})\times U_{r, s}(\mathfrak{g})$
$\langle~,~\rangle_{U}$ is $\mathrm{ad_{l}}$-invariant, that is,
$\langle \mathrm{ad_{l}}(a)b, c\rangle_{U}=\langle b,
\mathrm{ad_{l}}(S(a))c\rangle_{U}$ for $a, b, c\in U_{r,
s}(\mathfrak{g})$.
\end{upshape}
\end{dingli}
By Theorem 2.14 (\cite{{BGH1}}),
\begin{displaymath}
\langle a, b\rangle_{U}=0,\quad a\in U^{-\sigma}_{r, s}(\mathfrak{n}^-),\quad
b\in U^{\delta}_{r, s}(\mathfrak{n}),\quad \sigma, \delta\in Q^+,
\quad\sigma\neq\delta.
\end{displaymath}
\begin{dingli}[\cite{BGH1}]\label{BGH1Th2.14}
\begin{upshape}
For $\beta\in Q^+$, the skew pairing $\langle~,~\rangle$ on
$\mathcal{B}^{\prime-\beta}\times\mathcal{B}^{\beta}$ is nondegenerate.
\end{upshape}
\end{dingli}
\begin{yinli}
\begin{upshape}
If $\mu, \mu_{1}, \nu, \nu_{1}\in Q^+$,~then
\begin{displaymath}
\langle U^{-\nu}(\mathfrak{n}^-)U_{0}U^{\mu}(\mathfrak{n}^+),
U^{-\nu_{1}}(\mathfrak{n}^-)U_{0}U^{\mu_{1}}
(\mathfrak{n}^+)\rangle_{U}=0
\Longleftrightarrow\mu=\nu_{1}, \nu=\mu_{1}.
\end{displaymath}
\end{upshape}
\end{yinli}

Define a group homomorphism $\chi_{\eta,
\phi}: Q\times Q\rightarrow\mathbb{K}^\times$ as follows
$$
\chi_{\eta,\phi}(\eta_{1},\phi_{1})=\langle\omega^\prime_{\eta},
\omega_{\phi_{1}}\rangle\langle\omega^\prime_{\eta_{1}},
\omega_{\phi}\rangle,$$
where $(\eta, \phi)\in Q\times Q$,~$(\eta_{1},\phi_{1})\in Q\times Q$,~$\mathbb{K}^\times=\mathbb{K}\backslash\{0\}$.

\begin{yinli}\label{chracter}
\begin{upshape}
Suppose $r^ks^{l}=1$ if and only if $k=l=0$. If $\chi_{\eta,
\phi}=\chi_{\eta^\prime, \phi^\prime}$, then $(\eta,
\phi)=(\eta^\prime, \phi^\prime)$.
\end{upshape}
\end{yinli}

\begin{proof}[\upshape\KAI Proof:]
\begin{upshape}
Let $\zeta=\sum\limits_{i=1}^{n}\zeta_i\alpha_i$.  By definition,
$$
\langle\omega^{'}_{\zeta},\omega_i\rangle=\left\{
       \begin{array}{rl}
          r^{2(\varepsilon_i,\zeta)}s^{2(\varepsilon_{i+1},\zeta)}, & i<n,\\[2pt]
          r^{2(\varepsilon_n,\zeta)}(rs)^{-\zeta_i},& i=n.
       \end{array}
    \right.
$$
$$
\langle\omega^{'}_i,\omega_{\zeta}\rangle^{-1}=\left\{
       \begin{array}{rl}
          r^{2(\varepsilon_{i+1},\zeta)}s^{2(\varepsilon_{i},\zeta)}, & i<n,\\[2pt]
          s^{2(\varepsilon_n,\zeta)}(rs)^{-\zeta_i}, & i=n.
       \end{array}
    \right.
$$
It is easy to see that the conclusion is obvious for case $i<n$. For case $i=n$,
\begin{eqnarray*}
\lefteqn{\chi_{\eta,\phi}(0,\alpha_i)=\langle\omega^{'}_{\eta},\omega_i\rangle
=r^{2(\varepsilon_i,\eta)}s^{2(\varepsilon_{i+1},\eta)}}
\hspace{1cm}\\
&=&\chi_{\eta^{'},\phi^{'}}(0,\alpha_i)=r^{2(\varepsilon_i,\eta^{'})}s^{2(\varepsilon_{i+1},\eta^{'})}.
\end{eqnarray*}
By comparing the exponent of $r$,~$s$, we get that
$(\varepsilon_i,\eta-\eta^{'})=0,~(\varepsilon_{i+1},\eta-\eta^{'})=0$. So, $\eta=\eta^{'}$, and
\begin{eqnarray*}
\lefteqn{\chi_{\eta,\phi}(\alpha_i,0)=\langle\omega^{'}_{i},\omega_{\phi}\rangle
=r^{-2(\varepsilon_{i+1},\phi)}s^{-2(\varepsilon_{i},\phi)}}
\hspace{1cm}\\
&=&\chi_{\eta^{'},\phi^{'}}(\alpha_i,0)=r^{-2(\varepsilon_{i+1},\phi^{'})}s^{-2(\varepsilon_{i},\phi^{'})}.
\end{eqnarray*}
By comparing  the exponent of $r$,~$s$, we have
$(\varepsilon_i,\phi-\phi^{'})=0,~(\varepsilon_{i+1},
\phi-\phi^{'})=0$, i.e., $\phi=\phi^{'}$.
\end{upshape}
\end{proof}

\begin{yinli}
\begin{upshape}
If $r^{k}s^{l}=1\Longleftrightarrow k=l=0$. Then Rosso form
$\langle~,~\rangle_{U}$ is nondegenerate on $U$.
\end{upshape}
\end{yinli}
\begin{proof}[\upshape\KAI Proof:]
\begin{upshape}
We only have to prove that
$u\in U^{-\nu}(\mathfrak{n}^-)U_{0}U^{\mu}(\mathfrak{n}^+)$, $
\langle u, v\rangle_{U}=0$ for all
$v\in U^{-\mu}(\mathfrak{n}^-)U_{0}U^{\nu}(\mathfrak{n}^+)$, then $u=0$.
Let $\mu\in Q^+$, $\{u^\mu_{1},~u^\mu_{2},\cdots
u^\mu_{d_{\mu}}\}$ is a basis of $U^{\mu}(\mathfrak{n}^+)$, $\dim
U^{\mu}(\mathfrak{n}^+)=d_{\mu}$. By Theorem~\ref{BGH1Th2.14}, we can take a dual basis of
$U^{-\mu}(\mathfrak{n}^-)$ as $\{v^\mu_{1}, v^\mu_{2}, \cdots
v^\mu_{d_{\mu}}\}$, that is, $\langle v^\mu_{i},
u^\mu_{j}\rangle=\delta_{ij}$. Set
$\{v^\nu_{i}\omega^\prime_{\eta}\omega_{\phi}u^\mu_{j}\big|1\leq
i\leq d_{\nu},~1\leq j\leq d_{\mu}\}$
is a basis of $ U^{-\nu}(\mathfrak{n}^-)U_0U^{\mu}(\mathfrak{n}^+)$. By definition of the Rosso form,
\begin{eqnarray*}
\langle v^\nu_{i}\omega^\prime_{\eta}\omega_{\phi}u^\mu_{j},
v^\mu_{k}\omega^\prime_{\eta_{1}}\omega_{\phi_{1}}u^\nu_{l}\rangle_{U}
&=&\langle\omega^\prime_{\eta},
\omega_{\phi_{1}}\rangle\langle\omega^\prime_{\eta_{1}},
\omega_{\phi}\rangle \langle v^\mu_{k}, u^\mu_{j}\rangle\langle
S^2(v^\nu_{i}), u^\nu_{l}\rangle\\
&=&\delta_{kj}\delta_{il}(rs^{-1})^{2(\rho,
\nu)}\langle\omega^\prime_{\eta},
\omega_{\phi_{1}}\rangle\langle\omega^\prime_{\eta_{1}},
\omega_{\phi}\rangle.
\end{eqnarray*}
Let $u=\sum\limits_{i, j, \eta, \phi}\theta_{i, j, \eta,
\phi}v^\nu_{i}\omega^\prime_{\eta}\omega_{\phi}u^\mu_{j},~
v=v^\mu_{k}\omega^\prime_{\eta_{1}}\omega_{\phi_{1}}u^\nu_{l},~
1\leq k\leq d_{\mu},~1\leq l\leq d_{\nu},~\eta_{1},
\phi_{1}\in Q,~\rho$ is a half sum of positive roots.~Because of $\langle u , v
\rangle_{U}=0$,~we have
\begin{equation}
\sum\limits_{\eta, \phi}\theta_{l, k, \eta, \phi}(rs^{-1})^{2(\rho,
\nu)}\langle\omega^\prime_{\eta},
\omega_{\phi_{1}}\rangle\langle\omega^\prime_{\eta_{1}},
\omega_{\phi}\rangle =0.
\end{equation}
This identity also can reformulate as $\sum\limits_{\eta, \phi}\theta_{l, k, \eta,
\phi}(rs^{-1})^{2(\rho, \nu)}\chi_{\eta, \phi}=0. $~
By Dedekind Theorem (please ref to \cite{{BAI}}),~$\theta_{l, k, \eta,
\phi}=0.~$So $u=0$.
\end{upshape}
\end{proof}

\section{Harish-Chandra homomorphism}
We suppose $r^ks^l=1$ if and only if $k=l=0$ from now on. Denote by $Z(U)$ the center of $U_{r, s}(\mathfrak{so}_{2n+1})$.
It is easy to see that $Z(U)\subset
U^{0}$. We also define an algebra automorphism $\gamma^{-\rho}: U_{0}\rightarrow U_{0}$ as
\begin{displaymath}
\gamma^{-\rho}(\omega_{\eta}^\prime\omega_{\phi})
=\varrho^{-\rho}(\omega_{\eta}^\prime\omega_{\phi})\omega_{\eta}^\prime\omega_{\phi}.
\end{displaymath}

\begin{dingyi}
\begin{upshape}
Harish-Chandra homomorphism $\xi: Z(U)\rightarrow U_{0}$ is the restricted map $\gamma^{-\rho}\pi\big|_{Z(U)}$,~
\begin{displaymath}
\gamma^{-\rho}\pi: U^{0}\rightarrow U_{0}\rightarrow U_{0},
\end{displaymath}
where $\pi : U^{0}\rightarrow U_{0}$ is the canonical projection.
\end{upshape}
\end{dingyi}

\begin{dingli}
\begin{upshape}
When $n$ is even, $\xi: Z(U)\rightarrow U_{0}$ is injective for $U_{r, s}(\mathfrak{so}_{2n+1})$.
\end{upshape}
\end{dingli}
\begin{proof}[\upshape\KAI Proof:]
\begin{upshape}
Note that $U^{0}=U_{0}\bigoplus K$,~where
$K=\bigoplus\limits_{\nu>0}U^{-\nu}(\mathfrak{n}^-)U_{0}U^{+\nu}(\mathfrak{n}^+)$,~$K$~
is the two-sided ideal in $U^{0}$ which is the kernel of $\pi$. As both $\pi$ and
$\gamma^{-\rho}$ are algebra homomorphism, $\xi$ is an algebra homomorphism. Assume that $z\in
Z(U)$ and $\xi(z)=0$. Writing $z=\sum\limits_{\nu\in Q^+}z_{\nu}$ with $z_{\nu}\in U^{-\nu}(\mathfrak{n}^-)U_{0}U^{+\nu}(\mathfrak{n}^+)$, we have $z_{0}=0$.
Fix any $z_{\nu}\neq 0$ minimal with the property that $\nu\in Q ^+\backslash 0 $. Also choose bases $\{y_{k}\}$ and $\{x_{l}\}$ for  $U^{-\nu}(\mathfrak{n}^-)$ and
$U^{+\nu}(\mathfrak{n}^+)$, respectively. We may write
$z_{\nu}=\sum\limits_{k , l}y_{k}t_{k , l}x_{l},~t_{k, l}\in
U_{0}$. Then
\begin{align*}
0&=e_{i}z-ze_{i}\\
 &=\sum\limits_{\gamma\neq\nu}(e_{i}z_{\gamma}-z_{\gamma}e_{i})+\sum\limits_{k, l}(e_{i}y_{k}-y_{k}e_{i})t_{k, l}x_{l}
 +\sum\limits_{k, l}y_{k}(e_{i}t_{k, l}x_{l}-t_{k, l}x_{l}e_{i}).
\end{align*}
Note that $e_{i}y_{k}-y_{k}e_{i}\in U^{-(\nu-\alpha_{i})}(\eta^-)U_{0}$,
and only
\begin{displaymath}
\sum\limits_{k, l}(e_{i}y_{k}-y_{k}e_{i})t_{k, l}x_{l}\in
U^{-(\nu-\alpha_{i})}(\mathfrak{n}^-)U_{0}U^{+\nu}(\mathfrak{n}^+).
\end{displaymath}
Therefore, we have
\begin{displaymath}
\sum\limits_{k, l}(e_{i}y_{k}-y_{k}e_{i})t_{k, l}x_{l}=0.
\end{displaymath}
By the triangular decomposition of $U$ and the fact that $\{x_{l}\}$ is a basis of $U^{+\nu}(\mathfrak{n}^+)$,~
we get $\sum\limits_{k}e_{i}y_{k}t_{k, l}=
\sum\limits_{k}y_{k}e_{i}t_{k, l}$ for each $l$ and $i$. Now we fix $l$ and consider the irreducible module $L(\lambda)$ for
$\lambda\in\Lambda^+$.
Let $v_{\lambda}$ be the
highest weight vector of $L(\lambda)$, and set $m=\sum\limits_{k}y_{k}t_{k,
l}v_{\lambda}$.~Then for each $i$,~
\begin{displaymath}
e_{i}m=\sum\limits_{k}e_{i}y_{k}t_{k,
l}v_{\lambda}=\sum\limits_{k}y_{k}e_{i}t_{k, l}v_{\lambda}=0.
\end{displaymath}
Hence $m$ generates a proper submodule of $L(\lambda)$, which is  contradict  with the irreducibility of $L(\lambda)$,~so
\begin{displaymath}
m=\sum\limits_{k}y_{k}\varrho^{\lambda}(t_{k, l})v_{\lambda}=0.
\end{displaymath}
$\nu=k_{1}\alpha_{1}+k_{2}\alpha_{2}+\cdots+k_{n}\alpha_{n}\in Q^+$,
chose suitable
$\lambda=\lambda_{1}\varpi_{1}+\lambda_{2}\varpi_{2}+\cdots+\lambda_{n}\varpi_{n}$~ such that
$(\lambda,\alpha^{\vee}_{i})=\lambda_{i}\geq k_{i}$.
We might assume $\lambda_{i}>0$. By
\cite{{BGH2}} Theorem 2.12,~
$\sum\limits_{k}y_{k}\varrho^{\lambda}(t_{k, l})\in U^{-\nu}(\mathfrak{n}^-)
\mapsto\sum\limits_{k}y_{k}\varrho^{\lambda}(t_{k, l})v_{\lambda}$~
is injective,~so $\sum\limits_{k}y_{k}\varrho^{\lambda}(t_{k, l})=0$. As $\{y_{k}\}$ is a basis of $U^{-\nu}(\mathfrak{n}^-)$,~
$\varrho^{\lambda}(t_{k, l})=0$. Let $t_{k,l}=\sum_{\eta,\phi}
a_{\eta,\phi}\omega_{\eta}^{'}\omega_{\phi}$. We can claim: when $n$ is even, $\varrho^{\lambda}(t_{k,
l})=0$,~$\forall\lambda$,~then $t_{k,l}=0$. This claim also amounts to
\begin{equation}
\varrho^{\lambda}({\omega_{\eta}^{'}\omega_{\phi}})=\varrho^{\lambda}({\omega_{\zeta}^{'}\omega_{\psi}})
\Longleftrightarrow \eta=\zeta,~\phi=\psi.
\end{equation}
which is also equivalent to
\begin{equation}
\varrho^{\lambda}({\omega_{\eta}^{'}\omega_{\phi}})=1
\Longleftrightarrow \eta=0,~\phi=0.
\end{equation}
This can be proved by induction.
\end{upshape}
\end{proof}

\section{The image of $Z(U)$ under the Harish-Chandra homomorphism $\xi$}
Define an algebra homomorphism $\varrho^{0,
\lambda}:U_{0}\rightarrow\mathbb{K}$ for
$\lambda\in\Lambda$ by
$\varrho^{0, \lambda}
(\omega^\prime_{\eta}\omega_{\phi})=(rs^{-1})^{(\eta+\phi,
\lambda)}$.
And define an algebra homomorphism
$\varrho^{\lambda, \mu}:U^0\rightarrow\mathbb{K}$ for $\lambda, \mu\in\Lambda$ by
$
\varrho^{\lambda, \mu}(\omega^\prime_{\eta}\omega_{\phi})=
\varrho^{\lambda, 0}(\omega^\prime_{\eta}\omega_{\phi})\varrho^{0,
\mu}(\omega^\prime_{\eta}\omega_{\phi})
=\varrho^{\lambda}(\omega^\prime_{\eta}\omega_{\phi})\varrho^{0,
\mu}(\omega^\prime_{\eta}\omega_{\phi}).
$

\begin{yinli}\label{lemma:diffCharacter}
\begin{upshape}
Let $u=\omega^\prime_{\eta}\omega_{\phi},~\eta,~
\phi\in Q$ for arbitrary two-parameter quantum group of type $B_{n}$.~If $\varrho^{\lambda, \mu}(u)=1$ for all $\lambda, \mu\in\Lambda$,~ then $u=1$.
\end{upshape}
\end{yinli}

\begin{proof}[\upshape\KAI Proof:]
Denote ~$\eta=\eta_1\alpha_1+\eta_2\alpha_2+\cdots+\eta_n\alpha_n$,~$\phi=\phi_1\alpha_1+\phi_2\alpha_2+\cdots+\phi_n\alpha_n$,
$$\varrho^{0,\varpi_i}(\omega_{\eta}^{'}\omega_{\phi})=(rs^{-1})^{(\eta+\phi,\varpi_i)}=1\Longrightarrow \eta_i+\phi_i=0,\ \forall i,$$
$$\varrho^{\varpi_i,0}(\omega_{\eta}^{'}\omega_{-\eta})=(s^2r^{-2})^{(\alpha_i,\omega_i)\eta_i}=1\Longrightarrow \eta_i=0,\ \forall i.$$

This completes the proof.
\end{proof}

\begin{tuilun}
\begin{upshape}
If $u\in U_{0}, \varrho^{\lambda, \mu}(u)=0$ for all $(\lambda, \mu)\in \Lambda\times\Lambda$,~then $u=0$.
\end{upshape}
\end{tuilun}

\begin{proof}[\upshape\KAI Proof:]
\begin{upshape}
$
(\lambda, \mu)\mapsto\varrho^{\lambda,
\mu}(\omega^\prime_{\eta}\omega_{\phi}),~(\eta, \phi)\in
Q\times Q
$
is the character over group $\Lambda\times\Lambda$. It follows from  Lemma ~\ref{lemma:diffCharacter} that different
$(\eta,\phi)$ give rise to different characters. Suppose now that $u=\sum\theta_{\eta,
\phi}\omega^\prime_{\eta}\omega_{\phi}$, where $\theta_{\eta,
\phi}\in\mathbb{K}$. By assumption,
$
\sum\theta_{\eta, \phi}\varrho^{\lambda,
\mu}(\omega^\prime_{\eta}\omega_{\phi})=0
$
for all $(\lambda,\mu)\in\Lambda\times\Lambda$.~
By the linear independence of different characters,~$\theta_{\eta, \phi}=0,~u=0$.
\end{upshape}
\end{proof}

Let $U^0_{\flat}=\bigoplus\limits_{\eta\in
Q}\mathbb{K}\omega^\prime_{\eta}\omega_{-\eta}$, define the action of Weyl group on $U^0_{\flat}$ by
$
\sigma(\omega^\prime_{\eta}\omega_{-\eta})=
\omega^\prime_{\sigma(\eta)}\omega_{-\sigma(\eta)}$,~
for all $\sigma\in W, \eta\in Q$.

\begin{dingli}\label{importantIdentity}
\begin{upshape}
We have $\varrho^{\sigma(\lambda),
\mu}(u)=\varrho^{\lambda, \mu}(\sigma^{-1}(u))$ for arbitrary two-parameter quantum group of type $B_{n}$,~where $u\in
U^0_{\flat}, \sigma\in W$ and $\lambda, \mu\in\Lambda$.
\end{upshape}
\end{dingli}

\begin{proof}[\upshape\KAI Proof:]
\begin{upshape}
We only need to prove this theorem for $u=\omega^\prime_{\eta}\omega_{-\eta}$.
First, we claim that
$\varrho^{\sigma(\lambda), 0}(u)=\varrho^{\lambda,
0}(\sigma^{-1}(u))$ for all $\lambda\in\Lambda$,~$\sigma\in
W$. This can be checked for each $\sigma_i$ and $u=\omega^{\prime}_{\eta}\omega_{-\eta}$.

We prove $\varrho^{0, \mu}(u)=\varrho^{0, \mu}(\sigma^{-1}(u))$ in the following.
\begin{align*}
&\varrho^{0,\mu}(\omega^{\prime\eta_{1}}_{1}\omega^{\prime\eta_{2}}_{2}\omega^{-\eta_{1}}_{1}\omega^{-\eta_{2}}_{2})\\
=&\varrho^{(0, (\mu_{1}+\frac{1}{2}\mu_{2})\alpha_{1}+(\mu_{1}+\mu_{2})\alpha_{2})}(\omega^{\prime\eta_{1}}_{1}\omega^{\prime\eta_{2}}_{2}\omega^{-\eta_{1}}_{1}\omega^{-\eta_{2}}_{2})\\
=&(rs^{-1})^{((\mu_{1}+\frac{1}{2}\mu_{2})\alpha_{1}+(\mu_{1}+\mu_{2})\alpha_{2}, 0)}\\
=&1, \\
&\varrho^{0,\mu}(\sigma^{-1}(\omega^{\prime\eta_{1}}_{1}\omega^{\prime\eta_{2}}_{2}\omega^{-\eta_{1}}_{1}\omega^{-\eta_{2}}_{2}))
=1. \\
&\varrho^{\sigma(\lambda), \mu}(u)\\
=&\varrho^{\sigma(\lambda),0}(\omega^{\prime\eta_{1}}_{1}\omega^{\prime\eta_{2}}_{2}\omega^{-\eta_{1}}_{1}\omega^{-\eta_{2}}_{2})
\varrho^{0, \mu}(\omega^{\prime\eta_{1}}_{1}\omega^{\prime\eta_{2}}_{2}\omega^{-\eta_{1}}_{1}\omega^{-\eta_{2}}_{2})\\
=&\varrho^{\lambda, 0}(\sigma^{-1}(\omega^\prime_{\eta}\omega_{-\eta}))\varrho^{0, \mu}(\sigma^{-1}(\omega^\prime_{\eta}\omega_{-\eta}))\\
=&\varrho^{\lambda, \mu}(\sigma^{-1}(\omega^{\prime\eta_{1}}_{1}\omega^{\prime\eta_{2}}_{2}\omega^{-\eta_{1}}_{1}\omega^{-\eta_{2}}_{2}))\\
=&\varrho^{\lambda,\mu}(\sigma^{-1}(\omega^{\prime}_{\eta}\omega_{-\eta})),
\end{align*}
for all $u\in U^0_{\flat},~\varrho^{\sigma(\lambda),
\mu}(u)=\varrho^{\lambda,\mu}(\sigma^{-1}(u))$.
\end{upshape}
\end{proof}

Let us define
\begin{eqnarray*}
(U^0_{\flat})^W=\{u\in U^0_{\flat}\big|\sigma(u)=u, \forall\sigma\in
W\},
\end{eqnarray*}
$\kappa_{\eta, \phi}(\lambda, \mu)=\varrho^{\lambda,
\mu}(\omega^\prime_{\eta}\omega_{\phi})$ and $\kappa^i_{\zeta,
\psi}(\lambda, \mu)=\varrho^{\sigma_{i}(\lambda),
\mu}(\omega^\prime_{\zeta}\omega_{\psi}),~(\lambda, \mu)\in
\Lambda\times\Lambda$.

\begin{yinli}\label{importantlemma}
\begin{upshape}
As for two-parameters quantum groups of type $B_{2n}$, if $u\in U_{0}$,
$\varrho^{\sigma(\lambda),\mu}(u)=\varrho^{\lambda, \mu}(u)$, for $\forall\lambda$, $\mu\in\Lambda$ and $\sigma\in W$. Then $u\in (U^0_{\flat})^W$.
\end{upshape}
\end{yinli}
\begin{proof}[\upshape\KAI Proof:]
\begin{upshape}
If~$u=\Sigma_{\eta, \phi}\theta_{\eta,
\phi}\omega^{\prime}_{\eta}\omega_{\phi}\in U_{0}$,~such that
$\varrho^{\sigma(\lambda), \mu}(u)=\varrho^{\lambda,\mu}(u)$,~for all $\lambda, \mu\in\Lambda$ and $\sigma\in W$.~Identity
$
\sum\limits_{(\eta,\phi)}\theta_{\eta,\phi}\varrho^{\lambda,
\mu}(\omega^{\prime}_{\eta}\omega_{\phi})=\sum\limits_{(\zeta,
\psi)}\theta_{\zeta,\psi}\varrho^{\sigma_{i}(\lambda),
\mu}(\omega^{\prime}_{\zeta}\omega_{\psi})$
can be rewritten as
\begin{equation}\label{invariant:formula1}
\sum\limits_{(\eta, \phi)}\theta_{\eta, \phi}\kappa_{\eta, \phi}
=\sum\limits_{(\zeta, \psi)}\theta_{\zeta, \psi}\kappa^i_{\zeta,
\psi}.
\end{equation}
By Lemma \ref{lemma:diffCharacter}, both sides of identity (\ref{invariant:formula1}) are linear combinations of different characters on $\Lambda\times\Lambda$.
If $\theta_{\eta,\phi}\neq 0$, then $\kappa_{\eta,
\phi}=\kappa^i_{\zeta,\psi}$, for some $(\zeta,
\psi)\in\Lambda\times\Lambda$.
\par
Let $\eta=\sum_j\eta_j\alpha_j$,~$\phi=\sum_j\phi_j\alpha_j$,~$\zeta=\sum_j\zeta_j\alpha_j$,
 ~$\psi=\sum_j\psi_j\alpha_j$ for $1\leq j \leq n$, then
   \begin{eqnarray}
   \kappa_{\eta,\phi}(0,\varpi_j)
   &=&\varrho^{0,\varpi_j}(\omega_{\eta}^{'}\omega_{\phi})=(rs^{-1})^{(\eta+\phi,\varpi_j)}\\
   &=&\kappa_{\zeta,\psi}^{i}(0,\varpi_j)\\
   &=&\varrho^{0,\varpi_j}(\omega_{\zeta}^{'}\omega_{\psi})=(rs^{-1})^{(\zeta+\psi,\varpi_j)},
   \end{eqnarray}
therefore,
   \begin{equation}\label{invariant:formula2}
   \eta+\phi=\zeta+\psi.
   \end{equation}
By expanding $
    \kappa_{\eta,\phi}(\varpi_i,0)=\kappa_{\zeta,\psi}^{i}(\varpi_i,0),
$ we get
   \begin{eqnarray*}
   (s^2r^{-2})^{\eta_i-\zeta_i}
    =(s^2r^{-2})^{(\alpha_{i-1},-\alpha_i)\zeta_{i-1}+(\alpha_{i},-\alpha_i)\zeta_{i}+(\alpha_{i+1},-\alpha_i)\zeta_{i+1}}
    \varrho^{-\alpha_i}(\omega_{\zeta+\psi}).
   \end{eqnarray*}
   Comparing the index of both sides, it is not difficult to reach the
   conclusion $\zeta_i+\psi_i=0\quad (\forall~1\leq i \leq n)$,
   when $n$ is an even number. So $u\in U^0_{\flat}$. By Theorem \ref{importantIdentity},
$\varrho^{\lambda,\mu}(u)=\varrho^{\sigma(\lambda),\mu}(u)=\varrho^{\lambda,\mu}(\sigma^{-1}(u))$,
~for all ~$\lambda,\mu\in\Lambda$ and $\sigma\in W$. So ~$u=\sigma^{-1}(u),~u\in (U^0_{\flat})^W.$~
\end{upshape}
\end{proof}

\begin{dingli}
\begin{upshape}
For type $B_n$, $\varrho^{\lambda+\rho,\mu}(\xi(z))=\varrho^{\sigma(\lambda+\rho),\mu}(\xi(z))$.~
Especially, when $n$ is even, $\xi(Z(U))\subseteq
(U^0_{\flat})^W$.
\end{upshape}
\end{dingli}
\begin{proof}[\upshape\KAI Proof:]
\begin{upshape}
If $z\in Z(U)$, choose suitable $\lambda, \mu\in\Lambda$ such that  ~$(\lambda,
\alpha^{\vee}_{i})\geq 0$ for some fixed $i$. Denote by $v_{\lambda, \mu}\in M(\varrho^{\lambda, \mu})$ the highest weight vector.
\begin{eqnarray*}
zv_{\lambda, \mu}=\pi(z)v_{\lambda, \mu}=\varrho^{\lambda,
\mu}(\pi(z))v_{\lambda, \mu}=\varrho^{\lambda+\rho,
\mu}(\xi(z))v_{\lambda, \mu}.
\end{eqnarray*}
So $z$ acts on $M(\varrho^{\lambda, \mu})$ with scalar
$\varrho^{\lambda+\rho,
\mu}(\xi(z))$. By Lemma 2.6 in \cite{{BGH2}},
\vspace{0.3cm}
\begin{eqnarray}\label{including:formula}
e_{i}f_{i}^{(\lambda, \alpha_{i}^\vee)+1}v_{\lambda,
\mu}=[(\lambda, \alpha_{i}^\vee)+1]_{i}f_{i}^{(\lambda,
\alpha_{i}^\vee)}\frac {r_{i}^{-(\lambda,
\alpha_{i}^\vee)}\omega_{i}-s_{i}^{-(\lambda,
\alpha_{i}^\vee)}\omega^\prime_{i}}{r_{i}-s_{i}}v_{\lambda, \mu}.
\end{eqnarray}
\begin{eqnarray*}
\lefteqn{\left(r_i^{-(\lambda,\alpha_i^{\vee})}\omega_i-s_i^{-(\lambda,\alpha_i^{\vee})}\omega_i^{'}
         \right)v_{\lambda,\mu}}
\hspace{1cm}\\
&=&\left(r_i^{-(\lambda,\alpha_i^{\vee})}\varrho^{\lambda,\mu}(\omega_i)-
   s_i^{-(\lambda,\alpha_i^{\vee})}\varrho^{\lambda,\mu}(\omega_i^{'})
   \right)v_{\lambda,\mu}\\
&=&\left(r_i^{-(\lambda,\alpha_i^{\vee})}\varrho^{\lambda,0}(\omega_i)-s_i^{-(\lambda,\alpha_i^{\vee})}\varrho^{\lambda,0}(\varpi_i^{'})
   \right)\varrho^{0,\mu}(\omega_i)v_{\lambda,\mu}\\
&=&0.
\end{eqnarray*}
So $e_{j}f_{i}^{(\lambda, \alpha_{i}^\vee)+1}v_{\lambda, \mu}=0$.~
\begin{align*}\vspace{0.2cm}
zf_{i}^{(\lambda, \alpha_{i}^\vee)+1}v_{\lambda, \mu}&=\pi(z)f_{i}^{(\lambda, \alpha_{i}^\vee)+1}v_{\lambda, \mu}\\
&=\varrho^{\sigma_{i}(\lambda+\rho)-\rho, \mu}(\pi(z))f_{i}^{(\lambda, \alpha_{i}^\vee)+1}v_{\lambda, \mu}\\
&=\varrho^{\sigma_{i}(\lambda+\rho), \mu}(\xi(z))f_{i}^{(\lambda,
\alpha_{i}^\vee)+1}v_{\lambda, \mu}.
\end{align*}
$z$ acts on $M(\varrho^{\lambda, \mu})$ by scalar
$\varrho^{\sigma_{i}(\lambda+\rho), \mu}(\xi(z))$. So
\begin{equation}\label{including:formula2}
\varrho^{\lambda+\rho,
\mu}(\xi(z))=\varrho^{\sigma_{i}(\lambda+\rho), \mu}(\xi(z)).
\end{equation}\par
Now let us prove that (\ref{including:formula2}) is true for arbitrary $\lambda\in\Lambda$.
If $(\lambda, \alpha_{i}^\vee)=-1$,
then $\lambda+\rho=\sigma_{i}(\lambda+\rho)$,
~(\ref{including:formula2}) is true. If $(\lambda,
\alpha_{i}^\vee)<-1$, let
$\lambda^\prime=\sigma_{i}(\lambda+\rho)-\rho$, then
$(\lambda^\prime,\alpha_{i}^\vee)\geq0$,~$\lambda^\prime$~
such that (\ref{including:formula2}) exists.
Putting $\lambda^\prime=\sigma_{i}(\lambda+\rho)-\rho$~
on ~(\ref{including:formula2}), we can get the required formula.
Since the Weyl group is generated by $\sigma_{i}$'s, so
\begin{equation}
\varrho^{\lambda+\rho, \mu}(\xi(z))=\varrho^{\sigma(\lambda+\rho),
\mu}(\xi(z)),
\end{equation}
for all $\lambda, \mu\in\Lambda, \sigma\in
W$. By Lemma \ref{importantlemma}, the proof is complete.
\end{upshape}
\end{proof}

\section{$\xi$ is an algebra isomorphism for $U_{r,s}({\mathfrak{so}_{4n+1}})$ }

\begin{yinli}\label{isCenterElement}
\begin{upshape}
~$z\in
Z(U)$ if and only if $\mathrm{ad_{l}}(x)z=(\iota\circ\varepsilon(x))z$~
for all $x\in U$. $\varepsilon:U\rightarrow\mathbb{K}$ is the counit of $U$, $\iota:\mathbb{K}\rightarrow U$ is the unit map of $U$.
\end{upshape}
\end{yinli}
\begin{proof}[\upshape\KAI Proof:]
\begin{upshape}
$z\in Z(U)$ for all $x\in U$,
\begin{eqnarray*}
\mathrm{ad_{l}}(x)z=\sum\limits_{(x)}x_{(1)}z
S(x_{(2)})=z\sum\limits_{(x)}x_{(1)}S(x_{(2)})=(\iota\circ\varepsilon)(x)z.
\end{eqnarray*}
On contrary, if $\mathrm{ad_{l}}(x)z=(\iota\circ\varepsilon)(x)z$, for all $x\in U$, then
\begin{eqnarray*}
\omega_{i}z\omega^{-1}_{i}=\mathrm{ad_{l}}(\omega_{i})z=(\iota\circ\varepsilon)(\omega_{i})z=z.
\end{eqnarray*}
For the same reason, $\omega^{\prime}_{i}z(\omega_i^{\prime})^{-1}=z$, and
\begin{eqnarray*}
0=(\iota\circ\varepsilon)(e_{i})z=\mathrm{ad_{l}}(e_{i})z=e_{i}z+\omega_{i}z(-\omega_{i}^{-1})e_{i}=e_{i}z-ze_{i}
\end{eqnarray*}
\begin{eqnarray*}
0=(\iota\circ\varepsilon)(f_{i})z
=\mathrm{ad_{l}}(f_{i})z
=z(-f_{i}(\omega^{\prime}_{i})^{-1})+f_{i}z(\omega^{\prime}_{i})^{-1}
=(f_{i}z-zf_{i})(\omega^{\prime}_{i})^{-1}.
\end{eqnarray*}
So $z\in Z(u). $
\end{upshape}
\end{proof}

\begin{yinli}\label{findU}
\begin{upshape}
Suppose $\Psi: U^{-\mu}(\mathfrak{n}^-)\times U^{\nu}(\mathfrak{n}^+)\rightarrow\mathbb{K}$~
is a bilinear function, $(\eta, \phi)\in Q\times Q$. There exits a $u\in
U^{-\nu}(\mathfrak{n}^-)U_{0}U^{\mu}(\mathfrak{n}^+)$ such that
\begin{equation}\label{formulaFindU}
\langle u,
y\omega_{\eta_{1}}^\prime\omega_{\phi_{1}}x\rangle_{U}=\langle\omega_{\eta_{1}}^\prime,
\omega_{\phi}\rangle \langle\omega_{\eta}^\prime,
\omega_{\phi_{1}}\rangle\Psi(y, x),
\end{equation} for all $x\in U^{\nu}(\mathfrak{n}^+),~y\in U^{-\mu}(\mathfrak{n}^-)$.
\end{upshape}
\end{yinli}

\begin{proof}[\upshape\KAI Proof:]
\begin{upshape}
Let $\mu\in Q^+$, $\{u^\mu_{1}, u^\mu_{2},\cdots, u^\mu_{d_{\mu}}\}$
be a basis of $U^{\mu}(\mathfrak{n}^+)$, $\{v^\mu_{1}, v^\mu_{2},\cdots,
v^\mu_{d_{\mu}}\}$ be a basis of $U^{-\mu}(\mathfrak{n}^-)$ such that $\langle v^\mu_{i},u^\mu_{j}\rangle=\delta_{ij}$.
Let
\begin{eqnarray*}
u=\sum\limits_{i, j}\Psi(v^\mu_{j},
u^\nu_{i})v^\nu_{i}\omega^{\prime}_{\eta}\omega_{\phi}u^\mu_{j}(rs^{-1})^{-2(\rho,
\nu)}.
\end{eqnarray*}
It is easy to check that $u$ satisfies identity (\ref{formulaFindU}).
\end{upshape}
\end{proof}
Define a $U$-module structure on $U^*$ as $f\in U^*,
(x.f)(v)=f(\mathrm{ad}(S(x))v)$. Let $\beta:U\rightarrow U^*$,
\begin{eqnarray}
\beta(u)(v)= \langle u ,  v\rangle_{U}\quad u, v\in U.
\end{eqnarray}
By the non-degeneracy of $\langle$,~$\rangle_{U}$, $\beta$ is injective.

\begin{dingyi}
\begin{upshape}
$M$ is a finite dimensional $U$-module.
Define $C_{f,m}\in U^*$, $C_{f,m}(v)=f(v.m)$, $v\in U$, for $m\in M$ and $f\in M^*$.
\end{upshape}
\end{dingyi}

\begin{dingli}
\begin{upshape}
$M$ is a finite dimensional $U$-module and $M=\bigoplus\limits_{\lambda\in
wt(M)}M_{\lambda}$,
\begin{eqnarray*}
M_{\lambda}=\big\{m\in
M\big|(\omega_{i}-\varrho^\lambda(\omega_{i})) m=0, \quad
(\omega^\prime_{i}-\varrho^\lambda(\omega^\prime_{i})) m=0\big\}.
\end{eqnarray*} for $f\in M^*,m\in M$, there exists a unique $u\in U$
such that $C_{f, m}(v)=\langle u, v\rangle_{U}$ for all $v\in U$.
\end{upshape}
\end{dingli}

\begin{proof}[\upshape\KAI Proof:]
\begin{upshape}
From the non-degeneracy of $\langle~, \rangle_{U}$, it is easy to see that the existence of $u$ is unique. Suppose $m\in M_{\lambda}$,
\begin{eqnarray*}
v=y\omega^\prime_{\eta_{1}}\omega_{\phi_{1}}x, \quad x\in
U^{\nu}(\mathfrak{n}^+), \quad y\in U^{-\mu}(\mathfrak{n}^-), \quad (\eta_{1},
\phi_{1})\in Q\times Q,
\end{eqnarray*}
\begin{align*}
C_{f, m}(v)&=C_{f, m}(y\omega^\prime_{\eta_{1}}\omega_{\phi_{1}}x)=f(y\omega^\prime_{\eta_{1}}\omega_{\phi_{1}}x.m)\\
          &=\varrho^{\lambda+\nu}(\omega^\prime_{\eta_{1}}\omega_{\phi_{1}})f(y. x. m).
\end{align*}
~$(y, x)\mapsto f(y. x. m)$ is a bilinear function,
\begin{align*}
\varrho^{\lambda+\nu}(\omega^\prime_{\eta_{1}})=\langle\omega^\prime_{\eta_{1}},
\omega_{\lambda+\nu}\rangle^{-1}, \quad
\varrho^{\lambda+\nu}(\omega_{\phi_{1}})=\langle\omega^\prime_{\lambda+\nu},
\omega_{\phi_{1}}\rangle.
\end{align*} So,
$
C_{f, m}(v)=\langle\omega^\prime_{\eta_{1}},
\omega_{-\lambda-\nu}\rangle\langle\omega^\prime_{\lambda+\nu},
\omega_{\phi_{1}}\rangle f(y. x. m).
$
By Lemma \ref{findU}, there exists $u_{\nu\mu}$ such that $C_{f, m}(v)=\langle
u_{\nu\mu}, v \rangle_{U}$ for all $v\in
U^{-\mu}(\mathfrak{n}^-)U_{0}U^{\nu}(\mathfrak{n}^+)$.

Now for arbitrary $v\in U$, $v=\sum\limits_{\mu, \nu}v_{\mu\nu}$,
where $v_{\mu\nu}\in U^{-\mu}(\mathfrak{n}^-)U_{0}U^{\nu}(\mathfrak{n}^+)$.
Because $M$ is finite dimensional, there exists a finite set $\Omega$ such that  \vspace{0.2cm}
\begin{align*}
C_{f, m}(v)=C_{f, m}\Big(\sum\limits_{(\mu, \nu)\in\Omega
}v_{\mu\nu}\Big),~\forall~ v\in U.
\end{align*}
Let $u=\sum\limits_{(\mu, \nu)\in \Omega}u_{\nu\mu}$, then
\begin{align*}
C_{f, m}(v)&=C_{f, m}\Big(\sum\limits_{(\mu, \nu)\in\Omega}v_{\mu\nu}\Big)=\sum\limits_{(\mu, \nu)\in\Omega}C_{f, m}(v_{\mu\nu})\\
          &=\sum\limits_{(\mu, \nu)\in\Omega}\langle u_{\nu\mu}, v_{\mu\nu}\rangle_{U}=\sum\limits_{(\mu, \nu)\in\Omega}\langle
          u_{\nu\mu}, v\rangle_{U}\\
          &=\langle u,  v\rangle_{U}.
\end{align*}
\end{upshape}

This completes the proof.
\end{proof}

Suppose $M$ is a module in the category $\mathcal{O}$, and define a linear transformation
$\Theta: M\rightarrow M$,
\begin{align}
\Theta(m)=(rs^{-1})^{-2(\rho, \lambda)}m,
\end{align}
for all $m\in M_{\lambda},~\lambda\in\Lambda$.

\begin{yinli}
\begin{upshape}
$\Theta u=S^2(u)\Theta$, for all $u\in U$.
\end{upshape}
\end{yinli}

\begin{proof}[\upshape\KAI Proof:]
\begin{upshape}
We only need to check it for the generators $e_{i}$,~$f_{i}$,~$\omega_{i}$ and $\omega^\prime_{i}$. For any $m\in M_{\lambda}$,
we have
\begin{align*}
S^2(e_{i})\Theta(m)&=S^2(e_{i})(rs^{-1})^{-2(\rho, \lambda)}.m
                    =\omega_{i}^{-1}e_{i}\omega_{i}(rs^{-1})^{-2(\rho, \lambda)}.m\\
                   &=\frac{1}{\langle\omega^\prime_{i},\omega_i\rangle}(rs^{-1})^{-2(\rho, \lambda)}.m
                    =(rs^{-1})^{-2(\rho, \lambda+\alpha_{i})}e_{i}.m\\
                   &=\Theta(e_{i}.m),~1\leq i<n.
\end{align*}    
\begin{align*}               
S^2(f_{i})\Theta.  m&=S^2(f_{i})(rs^{-1})^{-2(\rho, \lambda)}.m
                     =\omega^\prime_{i}f_{i}\omega^{^\prime-1}_{i}(rs^{-1})^{-2(\rho, \lambda)}.m\\
                    &=\langle\omega^\prime_{i},\omega_i\rangle(rs^{-1})^{-2(\rho,\lambda)}f_i.m
                     =(rs^{-1})^{-2(\rho, \lambda-\alpha_{i})}f_{i}. m,\\
                    &= \Theta(f_{i}. m),~1\leq i<n.\\
S^2(e_n)\Theta(m)&=\omega^{-1}_n e_n \omega_n (rs^{-1})^{-2(\rho,\lambda)}.m\\
                 &=\frac{1}{\langle \omega^\prime_n,\omega_n \rangle}(rs^{-1})^{-2(\rho,\lambda)}e_n.m\\
                 &=(rs^{-1})^{-2(\rho,\lambda+\alpha_n)}e_n.m\\
                 &=\Theta(e_n.m).\\
S^2(f_n)\Theta(m)&=\omega^\prime_n f_n \omega^{\prime-1}_n (rs^{-1})^{-2(\rho,\lambda)}.m\\
                 &=\langle \omega^\prime_n,\omega_n \rangle(rs^{-1})^{-2(\rho,\lambda)}f_n.m\\
                 &=(rs^{-1})^{-2(\rho,\lambda-\alpha_n)}e_n.m\\
                 &=\Theta(f_n.m).
\end{align*}
\begin{eqnarray*}
\Theta(\omega_{i}. m)=\varrho^\lambda(\omega_{i})(rs^{-1})^{-2(\rho,\lambda)}m, &\quad&S^2(\omega_{i})\Theta. m=\varrho^\lambda(\omega_{i})(rs^{-1})^{-2(\rho, \lambda)}m, \\
\Theta(\omega^\prime_{i}.m)=\varrho^\lambda(\omega^\prime_{i})(rs^{-1})^{-2(\rho, \lambda)}m,
&\quad&
S^2(\omega^\prime_{i})\Theta. m=\varrho^\lambda(\omega^\prime_{i})(rs^{-1})^{-2(\rho, \lambda)}m, \\
\Theta(e_{i}. m)=S^2(e_{i})\Theta(m), &\quad&
\Theta(f_{i}. m)=S^2(f_{i})\Theta(m), \\
\Theta(\omega_{i}. m)=S^2(\omega_{i})\Theta(m),
&\quad&\Theta\omega^\prime_{i}. m=S^2(\omega^\prime_{i})\Theta(m).
\end{eqnarray*}
So $\Theta u=S^2(u)\Theta.$~
\end{upshape}
\end{proof}

For $\lambda\in\Lambda$, define $f_{\lambda}\in U^*,~
f_{\lambda}(u)=\mathrm{tr}_{L(\lambda)}(u\Theta),~u\in U$.

\begin{yinli}
\begin{upshape}
If $\lambda\in \Lambda^+\cap Q$, then $f_{\lambda}\in
\mathrm{Im}(\beta)$.~
\end{upshape}
\end{yinli}

\begin{proof}[\upshape\KAI Proof:]
\begin{upshape}
Let $k=\dim L(\lambda)$, $\{m_{i}\}$ is a basis of $L(\lambda)$, ~$\{f_{i}\}$ is the dual basis of $L(\lambda)^*$.
$$v\Theta(m_i)=\sum\limits_{j=1}^{k}f_j(v\Theta(m_j))m_j=\sum\limits_{j=1}^{k}C_{f_j,\Theta(m_j)}(v)m_j$$
\begin{align*}
f_{\lambda}(v)=\mathrm{tr}_{L(\lambda)}(v\Theta)=\sum\limits_{i=1}^k
C_{f_{i}, \Theta (m_i)}(v).
\end{align*}
By Theorem \ref{findU}, there is a $u_{i}\in U$ such that $C_{f_{i},\Theta
(m_i)}(v)=\langle u_{i}$,~$v\rangle_{U}$ for all $1\leq i\leq k$. Let $u=\sum\limits_{i=1}^k u_{i}$,
\begin{align*}
\beta(u)(v)=\langle u,v\rangle_U=\sum\limits_{i=1}^k\langle u_{i}, v
\rangle_{U}=\sum\limits_{i=1}^k C_{f_{i}, \Theta(m_i)}(v)=f_{\lambda}(v).
\end{align*}
So $f_{\lambda}\in \mathrm{Im}(\beta)$.
\end{upshape}
\end{proof}

\begin{dingli}
\begin{upshape}   
For $\lambda\in\Lambda^+\cap Q$,
$z_{\lambda}:=\beta^{-1}(f_{\lambda})\in Z(U)$.~
\end{upshape}
\end{dingli}

\begin{proof}[\upshape\KAI Proof:]
\begin{upshape}
 For all $x\in U$,
\begin{align*}
(S^{-1}(x)f_{\lambda})(u)&=f_{\lambda}(\mathrm{ad}(x)u)\\
&=\mathrm{tr}_{L(\lambda)}\Big(\sum\limits_{(x)}x_{(1)}uS(x_{(2)})\Theta\Big)\\
&=\mathrm{tr}_{L(\lambda)}\Big(u\sum\limits_{(x)}S(x_{(2)})\Theta (x_{(1)})\Big)\\
&=\mathrm{tr}_{L(\lambda)}\Big(u\sum\limits_{(x)}S(x_{(2)})S^2(x_{(1)})\Theta \Big)\\
&=\mathrm{tr}_{L(\lambda)}\Big(uS\Big(\sum\limits_{(x)}S(x_{(1)})x_{(2)}\Big)\Theta \Big)\\
&=(\iota\circ\varepsilon)(x)\mathrm{tr}_{L(\lambda)}(u\Theta)=(\iota\circ\varepsilon)(x)f_{\lambda}(u).
\end{align*}
Replacing $x$ with $S(x)$ in the above identity and noticing that $\varepsilon\circ S=\varepsilon$, we can get
$xf_{\lambda}=(\iota\circ\varepsilon)(x)f_{\lambda}$. Indeed, we have
\begin{align*}
xf_{\lambda}=x\beta(\beta^{-1}(f_{\lambda}))=\beta(\mathrm{ad}(S(x))\beta^{-1}(f_{\lambda})), \\
(\iota\circ\varepsilon)(x)f_{\lambda}=(\iota\circ\varepsilon)(x)\beta(\beta^{-1}(f_{\lambda}))
=\beta((\iota\circ\varepsilon)(x)\beta^{-1}(f_{\lambda})).
\end{align*}
\par
Because $\beta$ is injective,
$\mathrm{ad}(S(x))\beta^{-1}(f_{\lambda})=(\iota\circ\varepsilon)(x)\beta^{-1}(f_{\lambda})$.
Since $\varepsilon\circ S^{-1}=\varepsilon$,
 replacing $x$ with $S^{-1}(x)$ in the above formula, we get
\begin{align*}
\mathrm{ad}(x)\beta^{-1}(f_{\lambda})=(\iota\circ\varepsilon)(x)\beta^{-1}(f_{\lambda}),
~\text{for all}~x\in U.
\end{align*}
By Lemma \ref{isCenterElement}, $\beta^{-1}(f_{\lambda})\in Z(U)$.
\end{upshape}
\end{proof}

\begin{dingli}
\begin{upshape}
As for two-parameter quantum groups of type $B_{2n}$, if  $r^{k}s^l=1$ $\iff$ $k=l=0$, then
 $\xi: Z(U)\rightarrow(U_{\flat}^0)^W$ is an algebra isomorphism.
\end{upshape}
\end{dingli}

\begin{proof}[\upshape\KAI Proof:]
\begin{upshape}
$z_{\lambda}=\beta^{-1}(f_{\lambda}),
\text{for}~\lambda\in\Lambda^+\cap Q$
\begin{align*}
z_{\lambda}=\sum\limits_{\nu\geq 0}z_{\lambda, \nu},\quad
z_{\lambda, 0}=\sum\limits_{(\eta, \phi)\in Q\times Q}\theta_{\eta,
\phi}\omega^\prime_{\eta}\omega_{\phi}.
\end{align*}
Here $z_{\lambda, \nu}\in U^{-\nu}(\mathfrak{n}^-)U_{0}U^{\nu}(\mathfrak{n}^+)~
\text{and}~\theta_{\eta, \phi}\in\mathbb{K}$. $(\eta_{i},
\phi_{i})\in Q\times Q$,
\begin{align*}
\langle z_{\lambda},
\omega^\prime_{\eta_{i}}\omega_{\phi_{i}}\rangle_{U}=\langle
z_{\lambda, 0},
\omega^\prime_{\eta_{i}}\omega_{\phi_{i}}\rangle=\sum\limits_{\eta,
\phi}\theta_{\eta, \phi}\langle\omega^\prime_{\eta_{i}},
\omega_{\phi}\rangle\langle\omega^\prime_{\eta},
\omega_{\phi_{i}}\rangle.
\end{align*}\par
On the other hand,
\begin{align*}
\langle z_{\lambda},
\omega^\prime_{\eta_{i}}\omega_{\phi_{i}}\rangle_{U}&=\beta(z_{\lambda})(\omega^\prime_{\eta_{i}}\omega_{\phi_{i}})
=f_{\lambda}(\omega^\prime_{\eta_{1}}\omega_{\phi_{i}})=\mathrm{tr}_{L(\lambda)}(\omega^\prime_{\eta_{i}}\omega_{\phi_{i}}\Theta)\\
&=\sum\limits_{\mu\leq
\lambda}\dim(L(\lambda)_{\mu})(rs^{-1})^{-2(\rho, \mu)}\varrho^\mu(\omega^\prime_{\eta_{i}}\omega_{\phi_{i}})\\
&=\sum\limits_{\mu\leq
\lambda}\dim(L(\lambda)_{\mu})(rs^{-1})^{-2(\rho,
\mu)}\langle\omega^\prime_{\eta_{i}},
\omega_{-\mu}\rangle\langle\omega^\prime_{\mu},
\omega_{\phi_{i}}\rangle.
\end{align*}
It can be rewritten as
\begin{align*}
\sum\limits_{(\eta, \phi)}\theta_{\eta, \phi}\chi_{\eta,
\phi}=\sum\limits_{\mu\leq\lambda}\dim(L(\lambda)_{\mu})(rs^{-1})^{-2(\rho,
\mu)}\chi_{\mu, -\mu}.
\end{align*}
So
\begin{equation*}\theta_{\eta, \phi}=
\begin{cases}\dim(L(\lambda)_{\mu})(rs^{-1})^{-2(\rho, \mu)},&\eta+\phi=0, \\
0,&\text{otherwise. }
\end{cases}
\end{equation*}\par
\begin{equation}
z_{\lambda, 0}=\sum\limits_{\mu\leq\lambda}(rs^{-1})^{-2(\rho, \mu)}\dim(L(\lambda)_{\mu})\omega^\prime_{\mu}\omega_{-\mu},
\end{equation}
\begin{align*}
\xi(z_{\lambda})
&=\gamma^{-\rho}(z_{\lambda,0})\\
&=\sum\limits_{\mu\leq\lambda}(rs^{-1})^{-2(\rho,\mu)}\dim(L(\lambda)_{\mu})
  \varrho^{-\rho}(\omega^{\prime}_{\mu}\omega_{-\mu})\omega^{\prime}_{\mu}\omega_{-\mu}\\
&=\sum\limits_{\mu\leq\lambda}\dim(L(\lambda)_{\mu})\omega^\prime_{\mu}\omega_{-\mu}.
\end{align*}

We only need to prove $(U_{\flat}^0)^W\subseteq\xi(Z(U)) $. Let
$\lambda\in\Lambda^+\cap Q$, define \vspace{0.2cm}
\begin{align}
av(\lambda)=\frac{1}{|W|}\sum\limits_{\sigma\in
W}\sigma(\omega^\prime_{\lambda}\omega_{-\lambda})=\frac{1}{|W|}\sum\limits_{\sigma\in
W}\omega^\prime_{\sigma(\lambda)}\omega_{-\sigma(\lambda)}.
\end{align}
For arbitrary $\eta\in Q$, there exists $\sigma\in W$ such that $\sigma(\eta)\in\Lambda^+\cap Q$, so
$\big\{av(\lambda)\big|\lambda\in\Lambda^+\cap Q\big\}$ is a basis of $(U_{\flat}^0)^W$. Right now we only need to prove $av(\lambda)\in
\mathrm{Im}(\xi)$. By induction on the height of $\lambda$, when $\lambda=0$, $av(\lambda)=1=\xi(1)$,
since $\dim L(\lambda)_{\mu}=\dim L(\lambda)_{\sigma(\mu)}$, for all $\sigma\in W$, $\dim L(\lambda)_{\lambda}=1$.
\begin{align*}
\xi(z_{\lambda})=|W|av(\lambda)+|W|\sum \dim
(L(\lambda)_{\mu)})av(\mu),
\end{align*}
where $\mu<\lambda ,~\mu\in\Lambda^+\cap Q$. By induction, $av(\lambda)\in \mathrm{Im}(\xi)$. So
$(U_{\flat}^0)^W\subseteq\xi(Z(U))$.
\end{upshape}
\end{proof}

\bibliographystyle{amsalpha}

\end{document}